\documentclass[12pt,a4paper]{article}
\usepackage{graphicx}
\usepackage[]{amsmath,amssymb,amsthm}
\usepackage{paralist}
\usepackage{epstopdf}

\newtheorem{theorem}{Theorem}[section]
\newtheorem{proposition}[theorem]{Proposition}
\newtheorem{lemma}[theorem]{Lemma}
\newtheorem{corollary}[theorem]{Corollary}
\theoremstyle{definition}
\newtheorem{definition}[theorem]{Definition}
\theoremstyle{remark}

\newcommand{\R}{\mathbb{R}}

\newcommand{\Z}{\mathbb{Z}}

\newcommand{\Fix}{\textnormal{Fix}}

\newcommand{\Id}{\textnormal{Id}}

\newcommand{\B}{{\cal B}}

\begin{document}

\begin{center}
{\large{\bf 
Construction of heteroclinic networks in $\R^4$}}\\
\mbox{} \\
\begin{tabular}{cc}
{\bf Sofia B.\ S.\ D.\ Castro$^{\dagger}$} & {\bf Alexander Lohse$^{\ddagger,*}$} \\
{\small sdcastro@fep.up.pt} & {\small alexander.lohse@math.uni-hamburg.de}
\end{tabular}

\end{center}

\noindent $^{*}$ Corresponding author.

\noindent $^{\dagger}$ Faculdade de Economia and Centro de Matem\'atica, Universidade do Porto, Rua Dr.\ Roberto Frias, 4200-464 Porto, Portugal.

\noindent $^{\ddagger}$ Centro de Matem\'atica da Universidade do Porto, Rua do Campo Alegre 687, 4169-007 Porto, Portugal\footnote{Long-term address: Fachbereich Mathematik, Universit\"at Hamburg, Bundesstra{\ss}e 55, 20146 Hamburg, Germany}

\begin{abstract}
We study heteroclinic networks in $\R^4$, made of a certain type of simple robust heteroclinic cycle. In simple cycles all the connections are of saddle-sink type in two-dimensional fixed-point spaces. We show that there exist only very few ways to join such cycles together in a network and provide the list of all possible such networks in $\R^4$.
The networks involving simple heteroclinic cycles of type A are new in the literature and we describe the stability of the cycles in these networks: while the geometry of type A and type B networks is very similar, stability distinguishes them clearly.
\end{abstract}


\vspace{.3cm}

\vspace{2cm}

\section{Introduction}

The study of robust heteroclinic cycles and networks is well-established as an interesting subject in the scientific community. It was approached initially from the point of view of the existence of such heteroclinic objects\footnote{We point out that the purpose of the references in this section is to illustrate rather than to exhaust the existing bibliography on the subject.} as in Guckenheimer and Holmes \cite{GH} and dos Reis \cite{Reis}. Then, the study of the asymptotic stability of cycles led to results in Krupa and Melbourne \cite{KrupaMelbourne95a,KrupaMelbourne2004}. When addressing the stability of cycles in networks, it is clear that no individual cycle can be asymptotically stable and intermediate notions of stability appeared in Melbourne \cite{Melbourne1991}, Brannath \cite{Brannath}, Kirk and Silber \cite{KS}, Driesse and Homburg \cite{DriesseHomburg_a} and Podvigina and Ashwin \cite{PodviginaAshwin2011}. Of these {\em essential asymptotic stability} (e.a.s.)\footnote{In \cite{PodviginaAshwin2011} this is called {\em predominant asymptotic stability}.} is the strongest and the one we use.

In another spirit, an interest was taken in the dynamics near networks of various degrees of complexity, see e.g.\ Homburg and Knobloch \cite{HomburgKnobloch} or Postlethwaite and Dawes \cite{PostlethwaiteDawes}. Clearly, this must be done on a case-by-case basis although an efficient tool seems to be the use of symmetry to study a quotient network instead of the original one. Examples may be found in Aguiar {\em et al.} \cite{ACL2005}, Aguiar and Castro \cite{AguiarCastro}, and Castro {\em et al.} \cite{CLP}. 

The manifest interest in, and complexity of, the dynamics observed near networks provides a good reason for a systematic treatment of such objects, particularly in low dimensions where the analysis is tractable. 
At this point, we must distinguish between homoclinic and heteroclinic cycles: the former exhibit connections of a node to itself (or to another node in the same group orbit). Homoclinic cycles were systematically addressed by Sottocornola \cite{Sottocornola} and Podvigina \cite{Podvigina2013} (see also Homburg {\em et al.\ }\cite{HomburgKellnerGharaei}), while networks involving homoclinic cycles appear in the results of Driesse and Homburg \cite{DriesseHomburg} and Podvigina and Chossat \cite{PC2015}. We focus on a particular type of heteroclinic cycle, in the context of symmetry, made of what we call elementary building blocks, see section \ref{preliminaries}. These cycles are such that the equilibria lie on the coordinate axes which are fixed-point spaces, the connections occur in coordinate planes (also fixed-point spaces) and all the eigenvalues are real. See Figures \ref{B2B2} and \ref{B3B3} for an illustration.

A systematic study of heteroclinic networks in low dimension is useful in that it provides simple examples of interesting dynamics and a gallery of case studies. The study of bifurcations from heteroclinic networks, for instance, can greatly benefit from such a gallery. See Kirk {\em et al.} \cite{KPR} for an interesting example of such a phenomenon.
A first step towards such a systematic approach may be found in Castro and Lohse \cite{CastroLohse}. The construction of heteroclinic networks in the context of coupled cell systems has recently been addressed by Ashwin and Postlethwaite \cite{AP} and by Field \cite{Field2015}, with a different purpose. Given a heteroclinic network these authors find ways in which a coupled cell system may be constructed so that its dynamics exhibit the prescribed network. The present article provides a list of certain simple networks in $\R^4$, made of simple cycles in the sense of \cite{KrupaMelbourne2004}, satisfying Assumption A at the beginning of section \ref{construction}. Whether and how these can be realized in the context of coupled cell systems is the object of \cite{AP} and \cite{Field2015}. The simplex and cylinder methods of \cite[Section 2]{AP} are of particular interest since they ensure the networks have most of the desired properties.

This article's contribution consists of a complete list of a special type of heteroclinic network in dimension $4$, together with the study of the stability properties of each individual cycle. In particular, heteroclinic networks consisting of cycles of type $A$, so far absent from the literature, are considered. Although the geometry of networks of type A is similar to that of networks of type B, the stability properties of cycles are very different.

Stability of individual cycles for type A networks is very constrained. In fact, for the particular networks we consider, all cycles except at most one attract almost no points in their neighbourhood. 
For an e.a.s.\ network of type A consisting of two cycles, while one of the cycles attracts virtually nothing in its neighbourhood, the other cycle attracts initial conditions near either cycle in the network. When the e.a.s.\ network is made of three cycles, two of these attract almost nothing in their respective neighbourhoods while the third cycle attracts initial conditions near any connection in the network. Which cycle has the attracting properties depends on the linearization at the nodes of the network. This is reminiscent of the type B network in the set-up studied by Kirk and Silber \cite{KS} but very different from many cases addressed in \cite{CastroLohse}.

From the point of view of applications the interest in heteroclinic cycles and networks manifests itself in subjects such as neuroscience, geophysics, game theory and populations dynamics. Illustrations in the literature appear in the work of Chossat and Krupa \cite{ChossatKrupa2014}, Rodrigues \cite{Rodrigues}, Aguiar and Castro \cite{AguiarCastro} and Hofbauer and Sigmund \cite{HofbauerSigmund}, respectively.

By restricting to simple networks satisfying Assumption A in $\R^4$ we exclude networks of potential interest but beyond the scope of this paper. These include those where nodes have complex eigenvalues, see for instance Rodrigues and Labouriau \cite{RL2014}. Allowing for non-simple (again with respect to the definition in \cite{KrupaMelbourne2004}) cycles there are networks with more than one equilibrium on the same axis and on the same side of the origin. These can be obtained through the cylinder method of \cite{AP}. Another type of non-simple cycle occurs when the invariant space containing a connection is not a fixed-point space. Such examples appear for instance in replicator dynamics, see \cite{HofbauerSigmund}. In this case the existence of invariant subspaces is intrinsic to the dynamics rather than a consequence of symmetry. We can also obtain non-simple cycles by allowing equilibria outside the coordinate axes. Dynamics near such a network can be found in Kirk et al.\ \cite{Kirketal2010}.

The next section provides the necessary definitions and results for the remaining ones. Section \ref{construction} establishes simple results which lead to a complete list of a certain type of heteroclinic networks in $\R^4$. Examples of some of these are known in the literature. We provide the construction of the new possibilities in an appendix.
In Section \ref{stability}, we resort to the calculations of Podvigina and Ashwin \cite{PodviginaAshwin2011} to establish the stability indices for each cycle in the network. Joining two or more cycles in a network places restrictions on the possible values given in \cite{PodviginaAshwin2011}. Section \ref{conclusion} concludes.

\section{Preliminary results} \label{preliminaries}

Our concern is with vector fields in $\R^4$ described by a set of differential equations $\dot{y}=f(y)$, where $f$ is $\Gamma$-equivariant for some finite group $\Gamma \subset O(4)$, that is,
$$
f(\gamma .y)=\gamma. f(y), \; \; \; \forall \; \gamma \in \Gamma \; \; \forall \; y \in \R^4.
$$
A {\em heteroclinic cycle} consists of equilibria, also called nodes, $\xi_i$, $i=1, \hdots, m$, together with trajectories which connect them:
$$
[\xi_i \rightarrow \xi_{i+1}] \subset W^u(\xi_i) \cap W^s(\xi_{i+1}) \neq \emptyset \;\;\;\; (i=1,\hdots,m; \; \; \xi_{m+1}=\xi_1).
$$
If $\xi_i \in \Gamma\xi_1$ for all $i=2, \ldots ,m$, then we say the cycle is \emph{homoclinic}. When referring to a cycle as heteroclinic we always assume it is not homoclinic. Such a cycle connects at least two nodes in different group orbits. It is clear that group orbits of equilibria in a cycle always appear in the group orbit of the whole cycle. See \cite[subsection 2.5]{Podvigina2012} for an interesting discussion on the definition of a heteroclinic cycle.

Using the notion of a building block from \cite{PC2015} we describe the cycles we study as follows. A \emph{building block}, see \cite[p.\ 901]{PC2015}, is a sequence of equilibria and connections $[\xi_1 \to \ldots \to \xi_{m+1}]$ where $\xi_{m+1}=\gamma \xi_1$ for some $\gamma \in \Gamma$ and such that $\xi_i \notin \Gamma \xi_1$ for all $i=2, \ldots ,m$. We say a building block is \emph{elementary} if no two of the nodes $\xi_1, \ldots ,\xi_m$ belong to the same group orbit and $\gamma = \Id$.

A finite connected union of heteroclinic cycles is called a \emph{heteroclinic network}.

We assume that each connection $[\xi_i \rightarrow \xi_{i+1}]$ is of saddle-sink type in an invariant subspace in order to ensure robustness of the cycle.

Heteroclinic cycles are classified as {\em simple} if the connections between consecutive equilibria are contained in a two-dimensional subspace. We use the definition of Krupa and Melbourne \cite[p.\ 1181]{KrupaMelbourne2004}:
let $\Sigma_j \subset \Gamma$ be an isotropy subgroup and let $P_j=\textnormal{Fix}(\Sigma_j)$. Assume that for all $j=1,\hdots , m$ the connection $[\xi_j \rightarrow \xi_{j+1}]$ is a saddle-sink connection in $P_j$. Write $L_j=P_{j-1}\cap P_j$. A robust heteroclinic cycle $X \subset \R^4\backslash \{ 0\}$ is {\em simple} if 
\begin{enumerate}
	\item[(i)] $\dim P_j=2$ for each $j$;
	\item[(ii)] X intersects each connected component of $L_j\backslash \{ 0\}$ in at most one point.
\end{enumerate}

In most of the literature it seems to have been silently assumed that for simple cycles the linearization $\mathrm{d}f(\xi_j)$ has no double eigenvalues. 
We adopt this further assumption in referring to simple cycles and define simple networks accordingly.

\begin{definition}
A heteroclinic network is called {\em simple} if
\begin{enumerate}
	\item[(i)] it is made of simple cycles,
	\item[(ii)] it intersects each connected component of $L_j\backslash \{ 0\}$ in at most one point.
\end{enumerate}
\end{definition}

Simple cycles have been classified into types $A$, $B$, and $C$, both in the context of bifurcation of cycles (see Chossat {\em et al.\ }\cite{CKMS}) and in the context of their stability (see Krupa and Melbourne \cite{KrupaMelbourne2004} and Podvigina and Ashwin \cite{PodviginaAshwin2011}). A further type, $Z$, appears in Podvigina \cite{Podvigina2012} to include some cycles of types $B$ and $C$. We use the original classification into types $A$, $B$ and $C$, which we reproduce here from  \cite{KrupaMelbourne2004}.

\begin{definition}[Definition 3.2 in Krupa and Melbourne \cite{KrupaMelbourne2004}] 
Let $X \subset \R^4$ be a simple robust heteroclinic cycle.
\begin{enumerate}
	\item[(i)]  $X$ is of {\em type A} if $\Sigma_j \cong \Z_2$ for all $j$.
	\item[(ii)]  $X$ is of {\em type B} if there is a fixed-point subspace $Q$ with $\dim Q=3$, such that $X \subset Q$.
	\item[(iii)]  $X$ is of {\em type C} if it is neither of type $A$ nor of type $B$.
\end{enumerate}
\end{definition}
We use subscripts to indicate the number of distinct group orbits of equilibria in a cycle, and superscripts to denote whether $-\Id$ is an element of $\Gamma$ ($-$) or not ($+$). For example, a $B_3^-$ cycle has three (group orbits of) equilibria and $-\Id \in \Gamma$. It follows from the definition (see also \cite[Corollary 3.5]{KrupaMelbourne2004}) that the vector fields supporting cycles of type $A$ are equivariant under symmetry groups that do not possess a reflection, whereas those for cycles of types $B$ or $C$ do possess a reflection. In this context, reflection means reflection in a hyperplane.

The asymptotic stability of the cycles, dependent on the eigenvalues of the vector field at each equilibrium, has been studied by Krupa and Melbourne \cite{KrupaMelbourne95a,KrupaMelbourne2004}. Joining cycles in a network prevents each cycle from being asymptotically stable, calling for intermediate notions of stability. These have been introduced by Melbourne \cite{Melbourne1991}, Brannath \cite{Brannath}, Kirk and Silber \cite{KS}. More recently, Podvigina and Ashwin \cite{PodviginaAshwin2011} and Podvigina \cite{Podvigina2012} have revisited these notions.

As in Podvigina and Ashwin \cite{PodviginaAshwin2011}, we denote by $B_{\varepsilon}(X)$ an $\varepsilon$-neighbourhood of a (compact, invariant) set $X \subset \R^n$. We write $\B(X)$ for the basin of attraction of $X$, i.e.\ the set of points $x \in \R^n$ with $\omega(x) \subset X$. For $\delta>0$ the $\delta$-local basin of attraction is $\B_\delta(X):=\{x \in \B(X) \mid \phi_t(x) \in B_\delta(X) \: \forall t>0  \}$, where $\phi_t(.)$ is the flow generated by the system of equations. 

The following is the strongest intermediate notion of stability. We denote Lebesgue measure by $\ell(.)$.

\begin{definition}[Definition 1.2 in Brannath \cite{Brannath}]
A compact invariant set $X$ is called {\em essentially asymptotically stable (e.a.s.)} if it is asymptotically stable relative to a set $N \subset \R^n$ with the property that
\begin{align*}
\lim\limits_{\varepsilon \to 0} \frac{\ell(B_{\varepsilon}(X) \cap N)}{\ell(B_\varepsilon(X))} = 1.
\end{align*}
\end{definition}

In \cite{PodviginaAshwin2011} the concept of stability index is introduced. It provides a means of quantifying the attractiveness of a compact, invariant set $X$. See Definition 5 and section 2.3 in Podvigina and Ashwin \cite{PodviginaAshwin2011}.

\begin{definition}
For $x \in X$ and $\varepsilon, \delta >0$ define 
\begin{align*}
\Sigma_\varepsilon(x)&:=\frac{\ell(B_\varepsilon(x) \cap \B(X))}{\ell(B_\varepsilon(x))}, \qquad \Sigma_{\varepsilon,\delta}(x):=\frac{\ell(B_\varepsilon(x) \cap \B_\delta(X))}{\ell(B_\varepsilon(x))}.
\intertext{Then the {\em stability index} at $x$ with respect to $X$ is defined to be}
\sigma(x)&:=\sigma_+(x)-\sigma_-(x),
\intertext{where}
\sigma_-(x)&:= \lim\limits_{\varepsilon \to 0} \left[ \frac{\textnormal{ln}(\Sigma_\varepsilon(x) )}{\textnormal{ln}(\varepsilon)}  \right], \qquad \sigma_+(x):= \lim\limits_{\varepsilon \to 0} \left[ \frac{\textnormal{ln}(1-\Sigma_\varepsilon(x) )}{\textnormal{ln}(\varepsilon)}  \right].
\intertext{The convention that $\sigma_-(x)=\infty$ if $\Sigma_\varepsilon(x)=0$ for some $\varepsilon>0$ and $\sigma_+(x)=\infty$ if $\Sigma_\varepsilon(x)=1$ is introduced. Therefore, $\sigma(x) \in [-\infty, \infty]$. In the same way the {\em local stability index} at $x \in X$ is defined to be}
\sigma_{\textnormal{loc}}(x)&:=\sigma_{\textnormal{loc},+}(x)-\sigma_{\textnormal{loc},-}(x),
\intertext{with}
\sigma_{\textnormal{loc},-}(x):= \lim\limits_{\delta \to 0} &\lim\limits_{\varepsilon \to 0} \left[ \frac{\textnormal{ln}(\Sigma_{\varepsilon,\delta}(x))}{\textnormal{ln}(\varepsilon)}  \right], \: \sigma_{\textnormal{loc},+}(x):= \lim\limits_{\delta \to 0} \lim\limits_{\varepsilon \to 0} \left[ \frac{\textnormal{ln}(1-\Sigma_{\varepsilon,\delta}(x))}{\textnormal{ln}(\varepsilon)}  \right].
\end{align*}
\end{definition}

The stability index $\sigma(x)$ quantifies the local extent (at $x \in X$) of the basin of attraction of $X$. If $\sigma(x)>0$, then in a small neighbourhood of $x$ an increasingly large portion of points is attracted to $X$. If on the other hand $\sigma(x)<0$, then the portion of such points goes to zero as the neighbourhood shrinks.

Theorem 2.2 in \cite{PodviginaAshwin2011} establishes that both $\sigma(x)$ and $\sigma_{\textnormal{loc}}(x)$ are constant along trajectories. In order to characterize the attraction properties of a heteroclinic cycle in terms of the stability index we have to calculate only a finite number of indices. Podvigina and Ashwin \cite{PodviginaAshwin2011} denote the index along the trajectory leading to an equilibrium $\xi_j$ by $\sigma_j$. In heteroclinic networks there may be more than one such connection, which is why in such cases we differ from this notation by writing $\sigma_{ij}$ for the index along the trajectory from $\xi_i$ to $\xi_j$. Moreover, Theorem 2.4 in \cite{PodviginaAshwin2011} shows that the calculation of the indices can be simplified by restricting to a transverse section.

Local stability indices are related to essential asymptotic stability in the following way.

\begin{theorem}[Theorem 3.1 \cite{Lohse2015}] \label{local_index}
Let $X \subset \R^n$ be a heteroclinic cycle or network with finitely many equilibria and connecting trajectories. Suppose that the local stability index $\sigma_{\textnormal{loc}}(x)$ exists and is not equal to zero for all $x \in X$. Then $X$ is essentially asymptotically stable if and only if $\sigma_{\textnormal{loc}}(x)>0$ along all connecting trajectories.
\end{theorem}

Stability indices are always considered with respect to a set $X$. In the context of heteroclinic networks, this may be a single cycle or the entire network. We distinguish between these indices by referring to them as {\em c-} and {\em n-indices}, respectively.

\section{Construction of simple networks}\label{construction}
From now on we are concerned with simple networks in $\R^4$ satisfying the following
\paragraph{Assumption A:}
\begin{itemize}
 \item[(i)] All the cycles in the heteroclinic network have at least one elementary building block.
 \item[(ii)] All the equilibria in the heteroclinic network are on a coordinate axis.
\end{itemize}

Assumption A(i) means that symmetric images of connections and equilibria occur only as images of an entire cycle (or network). Assumption A(ii) is naturally satisfied when the simplex and cylinder methods are used to generate simple cycles. Cycles of types $B$ and $C$ follow directly from these methods, cycles of type $A$ can be obtained by adding some symmetry breaking terms.

In \cite[Proposition 3.1]{CastroLohse}, the authors have established some results concerning simple heteroclinic networks in $\R^4$ involving cycles of types $B$ and $C$. The results concern networks with at least one common connecting trajectory. We extend this study to address networks of type $A$. We also prove that having a common connecting trajectory is compulsory for networks satisfying Assumption A. The present section provides all the necessary results leading to this complete list. Note that Lemmas \ref{sym_group} to \ref{max_trajectory_per_node} do not require Assumption A.

The next lemma completes Proposition 3.1 in \cite{CastroLohse} which deals with networks made of only two cycles of types $B$ and $C$. For the sake of completion, we include the results of \cite{CastroLohse} in the statement of Lemma \ref{sym_group} which therefore informs on all three types of network.

\begin{lemma}\label{sym_group}
Cycles of type $A$ can only be part of a network involving other cycles of the same type. Cycles of types $B$ and $C$ can be part of networks involving other cycles of the same type and equal number of equilibria or by joining cycles of types $B_3^-$ and $C_4^-$.
\end{lemma}

\begin{proof}
According to the definition of the type of the cycles, the symmetry group of vector fields supporting cycles of type $A$ does not possess any reflections whereas that of vector fields supporting cycles of types $B$ and $C$ does. Therefore, cycles of type $A$ are compatible only with cycles of the same type. 

Concerning cycles of types $B$ and $C$, Krupa and Melbourne \cite[section 3.2]{KrupaMelbourne2004} have provided a description of the symmetry groups of supporting vector fields. These are as follows
\begin{center}
\begin{tabular}{|c|c|c|c|c|}
\hline
symmetry group & $\Z_2^3$ & $\Z_2^4$ & $\Z_2 \ltimes \Z_2^4$ & $\Z_2^4$ \\
\hline
cycle of type & $B_2^+$ & $B_3^-$ & $C_2^-$ & $C_4^-$ \\
\hline
\end{tabular}
\end{center}
concluding the proof.
\end{proof}

We now derive a few essential results that further limit the possibilities for heteroclinic networks in $\R^4$. These are then put together to obtain a complete list of networks in Theorem \ref{construct}.

\begin{lemma}\label{2nodes_lemma}
Consider a simple cycle with two nodes. These nodes belong to the same one-dimensional vector space.
\end{lemma}

\begin{proof}
Let $\xi_1$ and $\xi_2$ be the nodes of a simple cycle $[\xi_1 \rightarrow \xi_2 \rightarrow \xi_1]$. Then $[\xi_1 \rightarrow \xi_2] \subset P_2$ and $[\xi_2 \rightarrow \xi_1] \subset P_1$, where $P_i$, $i=1,2$ is a two-dimensional fixed-point space. Also, $\xi_1, \xi_2 \in P_1 \cap P_2 = L_1$ which is a one-dimensional vector space.
\end{proof}

\begin{lemma}\label{max_trajectory_per_node}
At each node of a simple network in $\R^4$ there exist at most three connecting trajectories (up to symmetry).
\end{lemma}

\begin{proof}
At each node of such a network there are four eigenvalues. Given that one of them is radial, there are three possible connecting directions corresponding to the remaining three eigenvalues.
\end{proof}

\begin{lemma}\label{max_nodes}
Up to symmetry, a simple network in $\R^4$ satisfying Assumption A does not involve more than four equilibria.
\end{lemma}

\begin{proof}
 Each connected component of $L_i\backslash \{ 0\}$ contains at most one equilibrium in the network because the network is simple. So suppose for some $i$ there are two equilibria $\xi_i, \xi_i^{'} \in L_i$ that are unrelated by symmetry. Then we have the following cases:
\begin{enumerate}
	\item[(a)] There is a connection $[\xi_i \to \xi_i^{'}]$ within a plane $P_{ij}$. Then $L_j$ is not an invariant subspace, so it cannot contain equilibria and the overall number of nodes does not exceed four.
	\item[(b)] There are connections $[\xi_i \rightarrow \xi_j]$ and $[\xi^{'}_i \rightarrow \xi_j]$ within $P_{ij}$. Then there exists $\gamma \in \Gamma$ such that $\xi_i^{'}=\gamma \xi_i$, because at $\xi_j$ the stable, unstable and transverse spaces are orthogonal complements.
	\item[(c)] There are connections $[\xi_i \rightarrow \xi_j]$ and $[\xi^{'}_i \rightarrow \xi_j']$ within $P_{ij}$. Then $\xi_j$ and $\xi^{'}_j$ are in different connected components of $L_j\backslash \{0\}$. It follows from \cite[Proposition 3.1]{KrupaMelbourne2004} that, if $P_{ij} = \Fix(\Sigma )$, then $\Z_2 \subset N(\Sigma)/\Sigma$ and $\xi^{'}_i \in \Gamma \xi_i$, contradicting our assumption.
	\item[(d)] There are connections $[\xi_i \rightarrow \xi_j]$ and $[\xi^{'}_i \rightarrow \xi_k]$ within planes $P_{ij}$ and $P_{ik}$. Then the respective other connections (i.e.\ those leading to $\xi_i$ and $\xi_i^{'}$) fall into one of the cases above.
\end{enumerate}
Note that for none of the cases the directions of the connections matter.
\end{proof}

\begin{corollary}\label{max_trajectory_network}
In $\R^4$, the maximum number of connecting trajectories in a simple network satisfying assumption A is six (up to symmetry).
\end{corollary}

\begin{lemma}\label{common_elements}
Two cycles in a simple network in $\R^4$ satisfying Assumption A have at least one node in common.
\end{lemma}
\begin{proof}
Assume there are two cycles without a common node. By Lemma \ref{max_nodes} each cycle has two equilibria, say $[\xi_1^A \to \xi_2^A \to \xi_1^A]$ and $[\xi_1^B \to \xi_2^B \to \xi_1^B]$. Then there must be a connection between the two cycles, say $[\xi_1^A \to \xi_1^B]$. Because of Lemma \ref{2nodes_lemma}, for $i=1,2$ the $\xi_i^A$ are on one axis, say $L_1$, and the $\xi_i^B$ are on another axis, say $L_4$. The existence of connections between $\xi_1^A$ and $\xi_2^A$ prevents the other axes from being fixed-point spaces. Therefore, the group orbit of $L_1$ is either itself or $L_4$.

If it is $L_1$, then the group image of $L_4$ is also $L_4$. Hence, $\Gamma$ fixes $L_1$ and $L_4$, so that neither is a fixed-point space, contradicting the definition of simple cycles. If it is $L_4$, then $\Gamma \xi_i^A = \xi_j^B$, so there is only one cycle.
\end{proof}

\begin{corollary}\label{connecting_trajectories}
Two cycles in a simple network in $\R^4$ satisfying Assumption A have at least one connecting trajectory in common.
\end{corollary}

\begin{proof}
By Lemma \ref{common_elements} there is a common node $\xi_i$. Suppose there is no common connection to or from $\xi_i$. Then there must be two incoming and two outgoing trajectories from $\xi_i$. This contradicts Lemma \ref{max_trajectory_per_node}.
\end{proof}

\begin{theorem}\label{construct}
In $\R^4$, the following is the complete list of simple heteroclinic networks satisfying assumption A:
\begin{itemize}
	\item  $(A_2,A_2)$; $(A_3,A_3)$; $(A_3,A_4)$; $(A_3,A_3,A_4)$;
	\item  $(B_2^+,B_2^+)$; $(B_3^-,B_3^-)$;
	\item  $(B_3^-,C_4^-)$; $(B_3^-,B_3^-,C_4^-)$.
\end{itemize}
\end{theorem}

\begin{proof}
Proposition 3.1 in \cite{CastroLohse} shows that the only networks made up of two cycles of type $B$ or $C$ are the ones listed here. 
Recall that the network $(B_2^+,B_2^+)$ is given in \cite{CastroLohse}, as is the network  $(B_3^-,C_4^-)$, while $(B_3^-,B_3^-)$ is studied by Kirk and Silber \cite{KS}. 
Note that the assumption of a common connecting trajectory in \cite[Proposition 3.1]{CastroLohse} is redundant given Corollary \ref{connecting_trajectories}. 
Although the assumption on non-existence of critical elements outside the network in \cite[Proposition 3.1]{CastroLohse} is stronger than Assumption A, we remark 
that the latter is sufficient for its proof.

In order to show that we do not have any other networks involving types $B$ and $C$, it follows from Lemma \ref{common_elements} that constructing a network from a cycle requires the addition of a connection to an already existing node. Then,
\begin{enumerate}
 \item[(i)] adding a cycle to the $(B_2^+,B_2^+)$ network requires an additional connecting trajectory at one of the nodes, contradicting Lemma \ref{max_trajectory_per_node}.
 \item[(ii)] adding a cycle to the $(B_3^-,B_3^-)$ or $(B_3^-,C_4^-)$ network, both of which already have five connecting trajectories, without contradicting Corollary \ref{max_trajectory_network}, leads to the $(B_3^-,B_3^-,C_4^-)$ network.
\end{enumerate}
The existence of the $(B_3^-,B_3^-,C_4^-)$ network is established by Brannath \cite{Brannath}. 

\medskip

From the previous results we know that, other than the networks listed, only networks joining cycles of type $A_2$ to cycles of type either $A_3$ or $A_4$ are possible. We begin by showing that there is an obstruction to the existence of such networks. 

Let $P_{ij} = \{ (x_1,x_2,x_3,x_4) \in \R^4 \mid x_k=0 \; \mbox{ for  } \; \; k\notin \{i,j\} \}$.
By Lemma \ref{2nodes_lemma} we may assume that the $A_2$ cycle is 
$[\xi_1 \rightarrow \xi_2 \rightarrow \xi_1]$ where $\xi_1, \xi_2 \in L_1$. Also, we may assume that $[\xi_1 \rightarrow \xi_2] \subset P_{12}$ and $[\xi_2 \rightarrow \xi_1] \subset P_{14}$. From Corollary \ref{connecting_trajectories}, assume that $[\xi_1 \rightarrow \xi_2] \subset P_{12}$ is the common connecting trajectory. 

Then the $A_3$ cycle connects $[\xi_1 \rightarrow \xi_2 \rightarrow \xi_3 \rightarrow \xi_1]$ where we can choose $[\xi_2 \rightarrow \xi_3]$ to be in $P_{13}$. Because $\xi_1$ and $\xi_2$ are on the same line, it follows that $[\xi_3 \rightarrow \xi_1]$ is also in $P_{13}$. But then the connections in $P_{13}$ are not of saddle-sink type and therefore, are not robust.

Concerning a connection to a cycle of type $A_4$, note that since $\xi_1$ and $\xi_2$ are on opposite sides of the origin, the connections between them in $P_{12}$ and $P_{14}$ must cross the $x_2$- and $x_4$-axis, respectively. There can be no equilibria on these axes as otherwise they would be invariant, preventing the existence of the connections in $P_{12}$ and $P_{14}$. 
Then the $A_4$ cycle must have two additional equilibria $\xi_3, \xi_4$ on the $x_3$-axis and  $[\xi_2 \rightarrow \xi_3], [\xi_4 \rightarrow \xi_1] \subset P_{13}$. Note that, for these cycles, the group orbit of the equilibria is non-trivial and $\gamma .\xi_i$ must belong to either of the axes already containing equilibria. This contradicts (ii) in the definition of simple networks.

\medskip

In order to show that the listed networks exist, we provide symmetry groups with the required fixed-point spaces for the existence of the networks of type $A$. In appendix A we construct vector fields supporting the networks of type $A$. A systematic approach to the construction of simple cycles, including homoclinic cycles, may be found in \cite{PC2015} where the quaternionic presentation is used as an alternative description of the symmetry groups.

\medskip

Since the symmetry group supporting a heteroclinic cycle of type $A$ has no reflections but still has subgroups isomorphic to $\Z_2$, we construct a group generated by rotations by $\pi$ on coordinate planes. This is multiplication by $-1$ of two of the four coordinates of $\R^4$ and is isomorphic to $\Z_2 = \langle -\Id \rangle
$. We look for groups generated by elements $\kappa_{ij}$ such that $P_{ij} = $Fix$(\langle \kappa_{ij} \rangle )$. 
For instance, we have
$$
\kappa_{12} = (\Id, R_{\pi_{34}})
$$
where $R_{\pi_{34}}$ is a rotation by $\pi$ on the plane 
$$
P_{34} = \{(0,0,x_3,x_4)\}.
$$
Then $\kappa_{12}(x_1,x_2,x_3,x_4) = (x_1,x_2,-x_3,-x_4)$.

We construct the groups so as to guarantee the existence of the invariant planes for the heteroclinic connections. These are 
\begin{itemize}
	\item  the $(A_2,A_2)$ network has one cycle with connections $[\xi_1 \rightarrow \xi_2 \rightarrow \xi_1]$ in  $P_{12} \cup P_{13}$ and the other with connections $[\xi_1 \rightarrow \xi_2 \rightarrow \xi_1]$ in $P_{12} \cup P_{14}$;
	\item  the $(A_3,A_3)$ network has one cycle with connections $[\xi_1 \rightarrow \xi_2 \rightarrow \xi_3 \rightarrow \xi_1]$ in  $P_{12} \cup P_{23} \cup P_{13}$ and the other with connections $[\xi_1 \rightarrow \xi_2 \rightarrow \xi_4 \rightarrow \xi_1]$ in  $P_{12} \cup P_{24} \cup P_{14}$;
	\item  the $(A_3,A_4)$ network has the $A_3$ cycle with connections $[\xi_1 \rightarrow \xi_2 \rightarrow \xi_3 \rightarrow \xi_1]$ in  $P_{12} \cup P_{23} \cup P_{13}$ and the $A_4$ cycle with connections $[\xi_1 \rightarrow \xi_2 \rightarrow \xi_3 \rightarrow \xi_4 \rightarrow \xi_1]$ in $P_{12} \cup P_{23} \cup P_{34} \cup P_{14}$;
	\item  the $(A_3,A_3,A_4)$ has, additionally to the previous, the $A_3$ cycle with connections $[\xi_1 \rightarrow \xi_2 \rightarrow \xi_4 \rightarrow \xi_1]$ in  $P_{12} \cup P_{24} \cup P_{14}$.
\end{itemize}

The groups for the respective networks are as follows:
\begin{itemize}
	\item  the $(A_2,A_2)$ network exists under the action of $\Gamma = \langle \kappa_{12}, \kappa_{13} \rangle$;
	\item  the $(A_3,A_3)$-, $(A_3, A_4)$- and $(A_3,A_3,A_4)$ networks exist under the action of $\Gamma = \langle \kappa_{12}, \kappa_{13}, \kappa_{34}  \rangle$.
\end{itemize}
We have $\kappa_{14} = \kappa_{12} \circ \kappa_{13}$, so $\Gamma = \langle \kappa_{12}, \kappa_{13} \rangle$ has $P_{12}$, $P_{13}$ and $P_{14}$ as invariant planes. Since $\kappa_{23}=\kappa_{13} \circ \kappa_{34}$ and $\kappa_{24}=\kappa_{14} \circ \kappa_{34}$, the group $\Gamma = \langle \kappa_{12}, \kappa_{13}, \kappa_{34} \rangle$ has as invariant planes all the coordinate planes.
\end{proof}
Considering the symmetry groups in the construction in the appendix, we find that $-\Id=\kappa_{12}\kappa_{34} \in \Gamma$ for the type $A$ cycles with three and four nodes. For $A$ cycles with two nodes $-\Id$ cannot be an element of $\Gamma$ since $\xi_1$ and $\xi_2$ are on the same axis and $-\Id$ maps $\xi_j$ to $-\xi_j$. In view of this, we write $A_2^+$, $A_3^-$ and $A_4^-$ from now on.\footnote{We thank an anonymous referee for this comment on the superscripts for $A$ cycles.} 

The networks made of cycles with the same number of nodes of type $A$ or $B$ only have an analogous geometry as seen in Figures \ref{B2B2} and \ref{B3B3}. The networks involving cycles of different types or of the same type but with a different number of nodes are represented in Figure \ref{B3C4}.

\begin{figure}[!htb]
\centerline{
\includegraphics[width=0.8\textwidth]{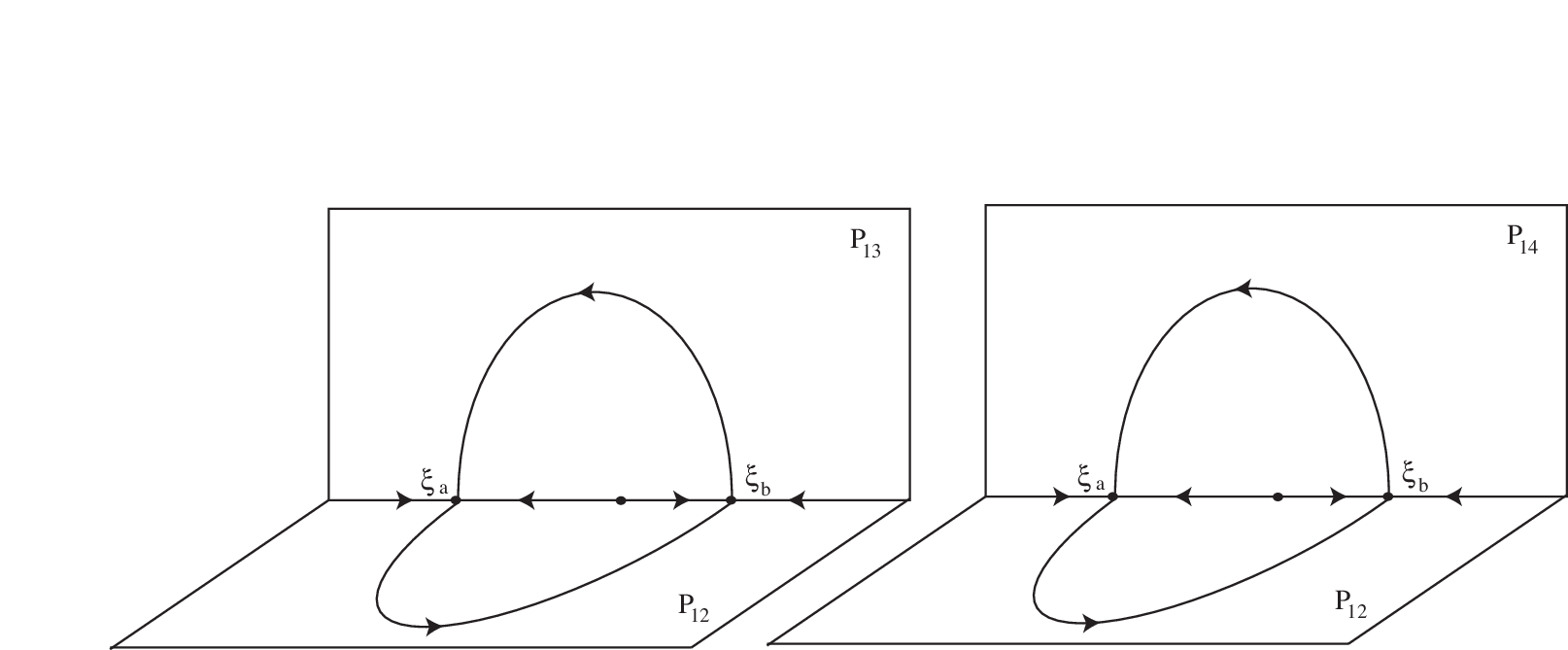}}
\caption{The cycles in the $(A_2^+,A_2^+)$- and the $(B_2^+,B_2^+)$ networks. The common connecting trajectory is in $P_{12}$.\label{B2B2}}
\end{figure}

\vspace{-1cm}
\begin{figure}[!htb]
\centerline{
\includegraphics[width=0.8\textwidth]{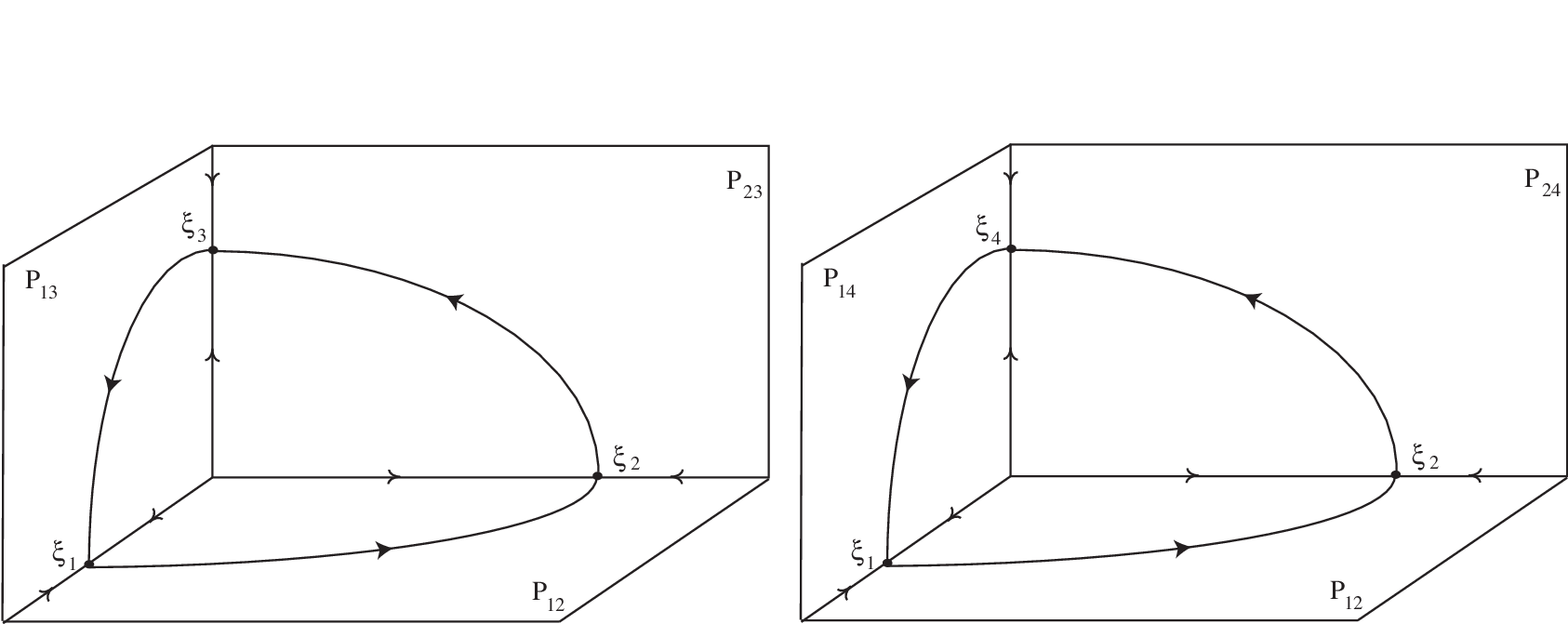}}
\caption{The cycles in the $(A_3^-,A_3^-)$- and the $(B_3^-,B_3^-)$ networks. The common connecting trajectory is in $P_{12}$.\label{B3B3}}
\end{figure}

\begin{figure}[!htb]
\centerline{
\includegraphics[width=0.75\textwidth]{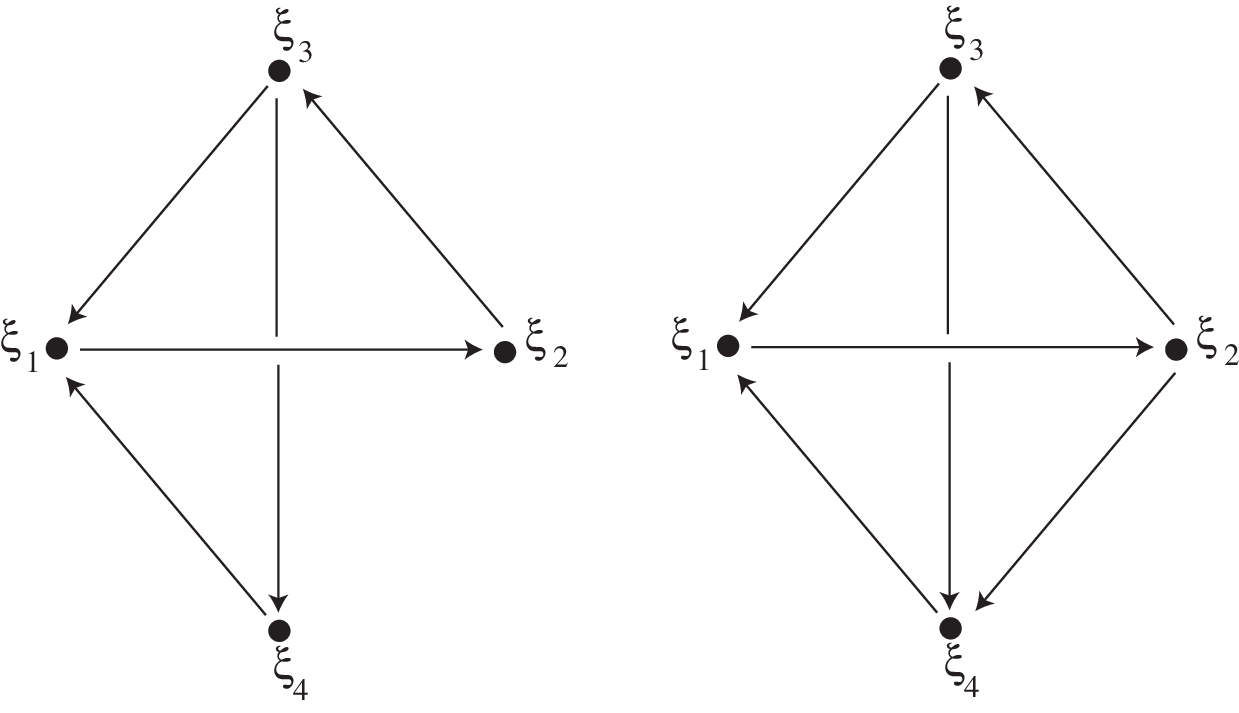}}
\caption{The $(B_3^-,C_4^-)$- or $(A_3^-,A_4^-)$- (left) and the $(B_3^-,B_3^-,C_4^-)$- or $(A_3^-,A_3^-,A_4^-)$ networks (right). On the right, the $A_4^-$ cycle ($C_4^-$) has two connecting trajectories in common with each of the $A_3^-$ cycles ($B_3^-$). The connection $[\xi_1 \rightarrow \xi_2]$ is common to all three cycles.
\label{B3C4}}
\end{figure}

\bigbreak

\section{Stability of cycles in type A networks}\label{stability}

In this section, we calculate the c-indices for each network of type A. We point out that, even though the geometry of the networks of type $A$ and of type $B$ is analogous, the stability indices behave differently.

For the rest of this paper, we assume that trajectories which leave a neighbourhood of the whole network do not come back. This ensures that the local and non-local stability indices coincide, as is implicit in \cite{PodviginaAshwin2011}.

In order to calculate the stability index, we linearize the vector field at each node. At a node $\xi_i$, there are four eigenvalues, often named $-r_i$ (radial), $-c_{ij}$ (contracting), $e_{ik}$ (expanding) and $t_{il}$ (transverse) for the role they play in the geometry of a cycle. We differ from this notation only for transverse eigenvalues $t_{il}$ that are expanding or contracting with respect to some other cycle in the network.\footnote{These are then called $e_{il}$ and $-c_{il}$, respectively.} The constants $r_i$, $c_{ij}$ and $e_{ik}$ are always assumed positive but $t_{il}$ can have either sign. For transverse eigenvalues $t_{il}$ with eigendirections away from the network, whenever possible we assume $t_{il}<e_{ik}$ so that, when the transverse eigenvalue is positive, it is weaker than the expanding one.
Define, following \cite{PodviginaAshwin2011}, $a_i=c_{ij}/e_{ik}$ and $b_i=-t_{il}/e_{ik}$.

Since the calculations in this section rely heavily on Theorem 4.1 in \cite{PodviginaAshwin2011}, we provide its statement together with relevant information for its understanding next.

We begin with the functions $h_{l,j}$ for $1\leq j\leq m$ and $l\leq j$  from \cite[p.\ 900]{PodviginaAshwin2011}:
\begin{align*}
 h_{j,j}(y)&=y,\\
 h_{l,j}(y)&=\begin{cases}+\infty &\textnormal{if} \quad a_l-b_l<0, \\ \frac{a_lh_{l+1,j}(y)-a_l+1}{a_l-b_l} &\textnormal{if} \quad 0<a_l-b_l<1,, \\ a_lh_{l+1,j}(y)-b_l &\textnormal{if} \quad a_l-b_l>1. \end{cases}
\end{align*}
Note that as usual the indices are to be understood modulo $m$, which is the number of equilibria in the cycle. With this we can reproduce the result on the stability indices of type $A$ cycles.

\bigbreak

\noindent {\bf Theorem (4.1 in \cite{PodviginaAshwin2011})}
{\em For a type $A$ cycle with $m$ equilibria the stability indices are as follows:
\begin{enumerate}
 \item[(a)] If $\rho>1$ and $b_j>0$ for all $j$, then $\sigma_j=+\infty$ for all $j$.
 \item[(b)] If $\rho>1$, $b_j>-1$ for all $j$ and $b_j<0$ for $j=J_1, \ldots J_L$, then
 \begin{equation*}
  \sigma_j=\min_{s=J_1,\ldots J_L}h_{\tilde{j},s}\left( -\frac{1}{b_s}\right)-1, \quad \textnormal{where} \quad \tilde{j}=\begin{cases}j &\textnormal{if} \quad j \leq s,\\ j-m &\textnormal{if} \quad j > s. \end{cases}
 \end{equation*}
 \item[(c)] If $\rho<1$ or there exists $j$ with $b_j<-1$, then $\sigma_j=-\infty$ for all $j$.
\end{enumerate}
}

When the cycles are part of a network the next result shows how their indices are further constrained. The next four results are proved for a generic connection $[\xi_i \to \xi_j]$ and their four relevant eigenvalues have $j$ as a first subscript. We preserve the notation of \cite{PodviginaAshwin2011} for the stability indices.

\begin{theorem}\label{minus_infinity}
Consider a network of type $A$ cycles with linearization at $\xi_j$ as above. Let $[\xi_i \rightarrow \xi_j]$ be a common connecting trajectory. At least one of the cycles has all c-indices equal to $-\infty$.
\end{theorem}

\begin{proof}
Assume, without loss of generality, that at $\xi_j$ the expanding eigenvalues are $e_{ja}$ and $e_{jb}$ satisfying $e_{ja}>e_{jb}$. Let $X_a$ and $X_b$ denote the cycles corresponding to the eigenvalues $e_{ja}$ and $e_{jb}$, respectively.

It suffices to show that for cycle $X_b$, we have case (c) of Theorem 4.1 in \cite{PodviginaAshwin2011}, hence $\sigma_j=-\infty$. In fact, for the cycle $X_b$ at $\xi_j$, we have $t_{jb}=e_{ja}$ so that $b_j = -e_{ja}/e_{jb} < -1$. 
\end{proof}

Note that for networks consisting of three cycles, two of these cycles have c-indices all equal to $-\infty$. 

On the side of stability, we have the following:

\begin{lemma}
For an $A$ cycle, all finite stability indices are non-negative.
\end{lemma}
\begin{proof}
In Theorem 4.1. of \cite{PodviginaAshwin2011} the only case with finite indices is (b), where $\sigma_{j,-}=0$ for all $j$ and thus $\sigma_j = \sigma_{j,+} \geq 0$.
\end{proof}
Taking into account the fact that for any cycle in $\R^4$, $\sigma_j=-\infty$ for some $j$ if and only if $\sigma_j=-\infty$ for all $j$ (Corollary 4.1 of \cite{PodviginaAshwin2011}), this gives
\begin{corollary}
Generically, an $A$ cycle with $\sigma_j>-\infty$ for some $j$, is e.a.s.
\end{corollary}

A simple observation of the computations involved in the calculation of the images of the maps $h_{l,j}$ provides the following result. Having established that if stability indices are not $-\infty$ they will be positive, we next address the question of whether they are finite or infinite.

\begin{lemma}\label{A-infinity}
For a type $A$ cycle in a simple network in $\R^4$,
\begin{enumerate}
 \item[(i)] $t_j>0 \Rightarrow \sigma_j<+\infty$,
 \item[(ii)] if $t_k \in (0,e_k) \; \forall k \neq j$, then: $\sigma_j=+\infty \; \Leftrightarrow \; t_j<-c_j$.
\end{enumerate}
\end{lemma}

\begin{proof}
\begin{enumerate}
\item[(i)] Let $t_j>0$. Then $b_j<0$, which excludes case (a) of Theorem 4.1 in \cite{PodviginaAshwin2011}. In case (c) all indices are equal to $-\infty$, so we look at case (b), where
\begin{equation*}
\sigma_j \leq h_{j,j}\left(-\frac{1}{b_j} \right)-1<+\infty.
\end{equation*}
\item[(ii)] If $t_j<0$ is the only negative transverse eigenvalue, the inequality $0<t_k<e_k$ for all $k \neq j$ puts us in case (b) again, where now
\begin{equation*}
\sigma_j = \min\limits_{s \neq j} h_{\tilde{j},s}\left(-\frac{1}{b_j} \right)-1.
\end{equation*}
For all $h_{\tilde{j},s}$ to be equal to $+\infty$, we need $a_j-b_j<0$, so $c_j+t_j<0$.\qedhere
\end{enumerate}
\end{proof}

So understanding a network of $A$ cycles means distinguishing e.a.s.\ cycles from those with all stability indices equal to $-\infty$. 
Notice that if all connections in a cycle have stability indices equal to $-\infty$ then almost all initial conditions near this cycle will either approach another cycle in the network or leave any neighbourhood of the network.
For an e.a.s.\ cycle, it is of interest to determine which indices are finite and which are equal to $+\infty$. To simplify notation, we implicitly assume that the stability index is finite when writing $\sigma_j>0$.

Following the terminology of \cite{PodviginaAshwin2011} we use the notation
\begin{align*}
\rho=\prod_j \rho_j \quad \text{with} \quad \rho_j=\min(a_j, 1+b_j).
\end{align*}
Note that $\rho$ is different for each of the cycles and $\rho>1$ is necessary for the respective cycle to be anything other than completely unstable.

\subsection{The network $(A_2^+,A_2^+)$}

Denote by $[\xi_a \rightarrow \xi_b]$ the common connecting trajectory of the network. Since this connection is contained in a plane, the linearization at $\xi_a$ and $\xi_b$ may be written, respectively, as
$$
\left\{ \begin{array}{l}
\dot{x_1} = -r_ax_1 \\
\dot{x_2} = e_{a2}x_2 \\
\dot{x_3} = -c_{a3}x_3 \\
\dot{x_4} = -c_{a4}x_4 ,
\end{array} \right. 
\;\;\; \mbox{ and   } \; \; \; \; 
\left\{ \begin{array}{l}
\dot{x_1} = -r_bx_1 \\
\dot{x_2} = -c_{b2}x_2 \\
\dot{x_3} = e_{b3}x_3 \\
\dot{x_4} = e_{b4}x_4 ,
\end{array} \right. 
$$
where all constants are positive. Denote the two cycles by $X_3$ and $X_4$. Without loss of generality, assume that at $\xi_b$ the positive eigenvalues $e_{b3}$ and $e_{b4}$ take trajectories to $X_3$ and $X_4$, respectively. Then the transverse eigenvalues for $X_3$ are $-c_{a4}$ and $e_{b4}$, while those for $X_4$ are $-c_{a3}$ and $e_{b3}$. The assumption that $e_{b3} > e_{b4}$ implies that $X_4$ has all stability indices equal to $-\infty$. The network is the union of the cycles depicted in Figure \ref{B2B2}.

The necessary data for Theorem 4.1 in \cite{PodviginaAshwin2011} is as follows:
\begin{align*}
\mbox{at } \xi_a, \ a_a&=\frac{c_{a3}}{e_{a2}} >0 \quad \text{and \quad}  b_a=\frac{c_{a4}}{e_{a2}} >0,\\[5pt]
\mbox{at } \xi_b, \ a_b&=\frac{c_{b2}}{e_{b3}} >0 \quad \text{and \quad} b_b=-\frac{e_{b4}}{e_{b3}} \in (-1,0).
\end{align*}

\begin{proposition}\label{indicesA2A2}
If $\rho>1$, the c-indices for $X_3$ in the $(A_2^+, A_2^+)$ network are:
\begin{align*}
\sigma_{ab} >0, \qquad
\sigma_{ba} = \begin{cases}
			+\infty &\Leftrightarrow \quad c_{a3}<c_{a4}\\
			>0 &\Leftrightarrow \quad c_{a3}>c_{a4}
		\end{cases}.
\end{align*}
\end{proposition}

\begin{proof}
There is only one negative transverse eigenvalue for $X_3$, which is $-c_{a4}$ at $\xi_a$. The transverse eigenvalue at $\xi_b$ is $e_{b4}$ and since we assume $e_{b3}>e_{b4}$, we are in case (ii) of Lemma \ref{A-infinity}, so $\sigma_{ba}=+\infty$ is decided by $c_{a3}\gtrless c_{a4}$. The remaining index is necessarily positive because it cannot be equal to $-\infty$.
\end{proof}

\subsection{The network $(A_3^-,A_3^-)$}

The network is geometrically as that in Figure \ref{B3B3}. We then have a cycle containing the nodes $\xi_1$, $\xi_2$ and $\xi_3$ with a common connection with another cycle with nodes  $\xi_1$, $\xi_2$ and $\xi_4$. We distinguish between these cycles by referring to the node they do not have in common. With this convention we talk about the $\xi_3$ cycle and the $\xi_4$ cycle.
We assume, as usual and without loss of generality, that $e_{23}>e_{24}$. According to Theorem \ref{minus_infinity}, at least one cycle has all c-indices equal to $-\infty$. Under the assumption that $e_{23}>e_{24}$, this is the $\xi_4$ cycle. In what follows, we calculate the expressions for the c-indices of the $\xi_3$ cycle, using Theorem 4.1 in Podvigina and Ashwin \cite{PodviginaAshwin2011}. 

The linearization at each node is the same as in  \cite{KS} (for a network of type $B$ cycles).

\begin{proposition}
Assume $\rho>1$. The $\xi_3$ cycle is e.a.s.\ and its c-stability indices are as follows:
\begin{align*}
\sigma_{12} &>0; \quad
\sigma_{31} = \begin{cases}
			+\infty &\Leftrightarrow \quad c_{14}>c_{13}\\
			>0 &\Leftrightarrow \quad c_{14}<c_{13}
		\end{cases};\\
\sigma_{23}&=\begin{cases}
			+\infty &\Leftrightarrow \quad (c_{14}>c_{13} \ \vee \ c_{32}<c_{34}) \ \wedge \ c_{34}>0\\
			>0 &\Leftrightarrow \quad (c_{14}<c_{13}  \ \wedge \ c_{32}>c_{34}) \ \vee \ c_{34}<0
		\end{cases}
\end{align*}
\end{proposition}

\begin{proof}
The required information at the nodes of the cycle is as follows. At $\xi_1$ we have
\begin{align*}
a_1&=\frac{c_{13}}{e_{12}} \quad \mbox{and} \quad b_1 = \frac{c_{14}}{e_{12}} >0.
\intertext{At $\xi_2$,}
a_2&=\frac{c_{21}}{e_{23}} \quad \mbox{and} \quad b_2 = -\frac{e_{24}}{e_{23}} \in (-1,0).
\intertext{At $\xi_3$,}
a_3&=\frac{c_{32}}{e_{31}} \quad \mbox{and} \quad b_3 = \frac{c_{34}}{e_{31}} .
\end{align*}
The proof has to proceed in two cases, according to the sign of $b_3$.

\paragraph{Case $b_3>0$:} This occurs when $c_{34}>0$. We then apply Theorem 4.1 (b) in \cite{PodviginaAshwin2011}. Along the connection $[\xi_3 \rightarrow \xi_1]$ we have
$$
\sigma_{31} = h_{1,2} \left(-\frac{1}{b_2} \right)-1.
$$
The value of $h_{1,2}$ depends on the sign of $a_1-b_1=\frac{c_{13}-c_{14}}{e_{12}}$. If $c_{13}-c_{14}<0$ then $h_{1,2}=+\infty$ and so is $\sigma_{31}$. Otherwise, $h_{1,2}$ is finite and $\sigma_{31}>0$.

For the c-index along $[\xi_1 \rightarrow \xi_2]$ we have
$$
\sigma_{12}=h_{2,2} \left(-\frac{1}{b_2} \right)-1=-\frac{1}{b_2}-1 = \frac{e_{23}}{e_{24}}-1 >0.
$$

Finally,
$$
\sigma_{23}=h_{3-3,2} \left(-\frac{1}{b_2} \right)-1=h_{0,2} \left(-\frac{1}{b_2} \right)-1.
$$

This is equal to $+\infty$ if $a_3-b_3<0$, which is the same as $c_{32}<c_{34}$. Otherwise $h_{0,2}$ is a function of $h_{1,2}$, giving the same conditions as for $\sigma_{31}$.
\paragraph{Case $b_3<0$:} This occurs when $c_{34}<0$. We then apply Theorem 4.1 (b) in \cite{PodviginaAshwin2011} but now have to account for two negative $b_j$'s. In this case, $\sigma_{ij}$ is the minimum of two images of $h_{l,j}$. Using Lemma \ref{A-infinity}, we know that all c-indices are positive and finite, with the possible exception of $\sigma_{ij}$ where the transverse eigenvalue at $\xi_j$ is negative. This is possible only for $\sigma_{31}$, so Lemma \ref{A-infinity} yields $\sigma_{31}=+\infty \Leftrightarrow c_{13}<c_{14}$.
\end{proof}

This is very different from the geometrically identical $(B_3^-,B_3^-)$ network for which various stability configurations (e.a.s.\ and non-e.a.s.) may occur, see section 5 in \cite{CastroLohse}.

\subsection{The network $(A_3^-,A_3^-,A_4^-)$}

Consider a network consisting of an $A_4^-$ cycle $[\xi_1 \to \xi_2 \to \xi_3 \to \xi_4 \to \xi_1]$ and two $A_3^-$ cycles, $[\xi_1 \to \xi_2 \to \xi_3 \to \xi_1]$ and $[\xi_1 \to \xi_2 \to \xi_4 \to \xi_1]$, which we call $\xi_3$ and $\xi_4$ cycle, respectively. Note that the geometry is not symmetric with respect to the $A_3^-$ cycles. This is the only way to build such a network, because there are only four eigendirections at each equilibrium.

We determine when each of the cycles is e.a.s.\ -- for ease of reference we have listed the quantities $a_j$ and $b_j$ in Table \ref{A3A3A4}. We use $\tilde{a}_j,\tilde{b}_j, \tilde{\sigma}_{ij}, \ldots$ for the $\xi_3$ cycle, plain letters for the $\xi_4$ cycle and $\bar{a}_j, \bar{b}_j, \bar{\sigma}_{ij}, \ldots$ for the $A_4^-$ cycle.

\linespread{1.6}
\begin{table}
\begin{center}
\begin{tabular}{|c|c|c|c|}
\hline 
\mbox{} & $\xi_3$ cycle & $\xi_4$ cycle & $A_4^-$ cycle \\
\hline
\hline
$\xi_1$ & $\tilde{a}_1 = \dfrac{c_{13}}{e_{12}}$; $\tilde{b}_1 = \dfrac{c_{14}}{e_{12}}$ & $a_1 = \dfrac{c_{14}}{e_{12}}$; $b_1 = \dfrac{c_{13}}{e_{12}}$ & $\bar{a}_1=\dfrac{c_{14}}{e_{12}}$; $\bar{b}_1=\dfrac{c_{13}}{e_{12}}$ \\[7pt]
\hline
$\xi_2$ & $\tilde{a}_2 = \dfrac{c_{21}}{e_{23}}$; $\tilde{b}_2 = -\dfrac{e_{24}}{e_{23}}$ & $a_2 = \dfrac{c_{21}}{e_{24}}$; $b_2 = -\dfrac{e_{23}}{e_{24}}$ & $\bar{a}_2=\dfrac{c_{21}}{e_{23}}$; $\bar{b}_2=-\dfrac{e_{24}}{e_{23}}$   \\[7pt]
\hline
$\xi_3$ & $\tilde{a}_3 = \dfrac{c_{32}}{e_{31}}$; $\tilde{b}_3 = -\dfrac{e_{34}}{e_{31}}$ & \mbox{}  & $\bar{a}_3=\dfrac{c_{32}}{e_{34}}$; $\bar{b}_3=-\dfrac{e_{31}}{e_{34}}$ \\[7pt]
\hline
$\xi_4$ & \mbox{}  & $a_4 = \dfrac{c_{42}}{e_{41}}$; $b_4 = \dfrac{c_{43}}{e_{41}}$   & $\bar{a}_4=\dfrac{c_{43}}{e_{41}}$; $\bar{b}_4=\dfrac{c_{42}}{e_{41}}$ \\[7pt]
\hline
\end{tabular}
\end{center}
\caption{The quantities $a_j$ and $b_j$ for cycles in the $(A_3^-, A_3^-, A_4^-)$ network. \label{A3A3A4}}
\end{table}
\linespread{1}

\begin{lemma}\label{A3A3A4-lem}
In the $(A_3^-,A_3^-,A_4^-)$ network
\begin{itemize}
\item[(a)] the $\xi_3$ cycle is e.a.s.\ if and only if $\tilde{\rho}>1$ and $e_{24}<e_{23}$ and $e_{34}<e_{31}$. If this is the case, then its stability indices are
	\begin{align*}
	\tilde{\sigma}_{31}&=\begin{cases}
			+\infty &\Leftrightarrow \quad c_{14}>c_{13}\\
			>0 &\Leftrightarrow \quad c_{14}<c_{13}
		\end{cases}\\
	\tilde{\sigma}_{12},\tilde{\sigma}_{23}&>0.
	\end{align*}
\item[(b)] the $\xi_4$ cycle is e.a.s.\ if and only if $\bar{\rho}>1$ and $e_{24}>e_{23}$. If this is the case, then its stability indices are
	\begin{align*}
	\sigma_{41},\sigma_{24}&=\begin{cases}
			+\infty  &\Leftrightarrow \quad c_{13}>c_{14}\\
			>0  &\Leftrightarrow \quad c_{13}<c_{14}
		\end{cases}\\
	\sigma_{12}&>0.
	\end{align*}
\item[(c)] the $A_4^-$ cycle is e.a.s.\ if and only if $\rho>1$ and $e_{24}<e_{23}$ and $e_{34}>e_{31}$. If this is the case, then its stability indices are
	\begin{align*}
	\bar{\sigma}_{41}&=\begin{cases}
			+\infty  &\Leftrightarrow \quad c_{13}>c_{14}\\
			>0  &\Leftrightarrow \quad c_{13}<c_{14}
		\end{cases}\\
	\bar{\sigma}_{34}&=\begin{cases}
			+\infty  &\Leftrightarrow \quad c_{13}>c_{14} \; \vee \; c_{42}>c_{43}\\
			>0  &\Leftrightarrow \quad c_{13}<c_{14} \; \wedge \; c_{42}<c_{43}
	\end{cases}\\
	\bar{\sigma}_{12}, \bar{\sigma}_{23}&>0.
	\end{align*}
\end{itemize}
\end{lemma}
\begin{proof}
In all cases, the conditions for e.a.s.\ are straightforward applications of Theorem 4.1 in \cite{PodviginaAshwin2011}: they are necessary and sufficient to have $\tilde{b}_j,b_j,\bar{b}_j>-1$ for all $j$ in the respective cycle. So we calculate the indices for each case, making use of Table \ref{A3A3A4}.

In case (a) we look at the $\xi_3$ cycle. Its only negative transverse eigenvalue is $-c_{14}$ at $\xi_1$. So Lemma \ref{A-infinity} immediately yields $\tilde{\sigma}_{12}, \tilde{\sigma}_{23}>0$, and $\tilde{\sigma}_{31}=+\infty$ if and only if $c_{14}>c_{13}$.

In case (b) the $\xi_4$ cycle has two negative transverse eigenvalues, thus $b_2<0$ is the only negative $b_j$. From Theorem 4.1 in \cite{PodviginaAshwin2011} we get
\begin{align*}
\sigma_{41}&=h_{1,2}\left(-\frac{1}{b_2}\right)-1,\\
\sigma_{12}&=h_{2,2}\left(-\frac{1}{b_2}\right)-1=\frac{e_{24}}{e_{23}}-1>0,\\
\sigma_{24}&=h_{0,2}\left(-\frac{1}{b_2}\right)-1.
\end{align*}
Now $h_{1,2}(.)$ is equal to $+\infty$ if $a_1-b_1=\frac{c_{14}-c_{13}}{e_{12}}<0$ . Otherwise it is a finite expression involving $h_{2,2}(.)$, which is finite. For $h_{0,2}(.)$ we have $a_2-b_2=\frac{c_{21}+e_{23}}{e_{24}}>0$, so it is a finite expression involving $h_{1,2}(.)$ and thus the reasoning from above applies.

For statement (c) and the $C_4^-$ cycle we have $\bar{b}_2, \bar{b}_3<0$ and thus
\begin{align*}
\bar{\sigma}_{12}&\leq h_{1,2}\left(-\frac{1}{\bar b_2}\right)-1=\frac{e_{23}}{e_{24}}-1<+\infty,\\
\bar{\sigma}_{23}&\leq h_{3,3}\left(-\frac{1}{\bar b_3}\right)-1=\frac{e_{34}}{e_{31}}-1<+\infty,\\[5pt]
\bar{\sigma}_{34}&=\min\left(h_{0,2}(-1/\bar b_2), h_{0,3}(-1/\bar b_3) \right)-1,\\
\bar{\sigma}_{41}&=\min\left(h_{1,2}(-1/\bar b_2), h_{1,3}(-1/\bar b_3) \right)-1.
\end{align*}
Conditions for the latter two indices to be finite follow in the same way as before.
\end{proof}

No two cycles can be e.a.s.\ simultaneously. This is clear even without lemma \ref{A3A3A4-lem} because they all share the trajectory $[\xi_1 \to \xi_2]$. We add to this

\begin{corollary}
As long as $\rho>1$ for all cycles, the $(A_3^-,A_3^-,A_4^-)$ network is not completely unstable.
\end{corollary}

\paragraph{The network $(A_3^-,A_4^-)$:}

This network may be thought of as obtained from the previous one by deletion of one of the $A_3^-$ cycles. If we keep the assumptions on the signs and magnitude of the transverse eigenvalues, the c-indices are the same as above. In this case, however, it is possible to make the network potentially more stable by choosing transverse eigenvalues to be negative whenever possible. Assume we delete the $\xi_4$ cycle, i.e.\ the connection $[\xi_2 \rightarrow \xi_4]$, to obtain the network $(A_3^-,A_4^-)$. We may now choose $t_2=e_{24}<0$.

Both remaining cycles are affected in the same way. For the $\xi_3$ cycle we may now have $\tilde{a}_2-\tilde{b}_2<0$ (if and only if $c_{21}<-e_{24}$), then $\tilde{\sigma}_{12}=+\infty$. Similarly, for the $A_4^-$ cycle we get $\bar{\sigma}_{12}=+\infty$ if and only if $c_{21}>-e_{24}$.

\section{Concluding remarks}\label{conclusion}
We have contributed to the study of a certain type of simple heteroclinic networks in $\R^4$, not just by providing a complete list of such simple networks, but by studying the stability of cycles in networks of type $A$. In the heteroclinic context, type $A$ networks have been overlooked. We show that even though their geometry resembles that of networks of other types, the stability of $A$ cycles can be very different.

The present list of simple heteroclinic networks provides convenient examples for extending research on heteroclinic dynamics to questions arising from bifurcation, noise addition and so on. This is beyond the scope of the present paper.

\paragraph{Acknowledgements:}

The authors are grateful to Peter Ashwin and Olga Podvigina for fruitful conversations.

The first author was partially supported by CMUP (UID/MAT/00144/2013),
which is funded by FCT (Portugal) with national (MEC) and European
structural funds through the programs FEDER, under the partnership
agreement PT2020.

The second author gratefully acknowledges support through the Nachwuchsfonds of the Mathematics Department in Hamburg and grant ``Incentivo/MAT/UI0144/2014'' of the FCT through CMUP, Portugal.

Some of this research took place during a visit of the first author to the Mathematics Department of the University of Hamburg whose hospitality is gratefully acknowledged.

\subsection*{A.1 \hspace{.2cm} The $(A_2^+,A_2^+)$ network}

From the proof of Theorem \ref{construct} we know that $\Gamma=\langle \kappa_{12}, \kappa_{13} \rangle \cong \Z_2^2$ supports $A_2^+$ cycles. Since $\kappa_{14}=\kappa_{12} \circ \kappa_{13}$, there are three planar fixed-point subspaces
\begin{eqnarray*}
P_{12} & = & \{(x_1,x_2,0,0)\} = \mbox{Fix}\langle \kappa_{12} \rangle \\
P_{13} & = & \{(x_1,0,x_3,0)\} = \mbox{Fix}\langle \kappa_{13} \rangle \\
P_{14} & = & \{(x_1,0,0,x_4)\} = \mbox{Fix}\langle \kappa_{14} \rangle,
\end{eqnarray*}
where $\langle \kappa_{1k} \rangle \simeq \Z_2$. The only fixed-point axis is
\begin{equation*}
L_1=\{(x_1,0,0,0)\} = \mbox{Fix} \langle \kappa_{12},\kappa_{13} \rangle = \Fix(\Gamma).
\end{equation*}
We can produce a heteroclinic cycle of type $A_2^+$ in a similar way to that of the construction of the $B_2^+$ cycle in \cite{CastroLohse}. A $\Gamma$-equivariant vector field is
\begin{eqnarray*}
\dot{x_1}& = & a_1x_1 + \sum_{i=1}^{4} b_{1i} x_i^2 + c_1x_1^3 \\
\dot{x_2}& = & a_2x_2 + \left(\sum_{i=1}^{4} b_{2i} x_i^2\right)x_2 + c_2x_2^2x_3x_4 \\
\dot{x_3}& = & a_3x_3 + \left(\sum_{i=1}^{4} b_{3i} x_i^2\right)x_3 + c_3x_2x_3^2x_4 \\
\dot{x_4}& = & a_4x_4 + \left(\sum_{i=1}^{4} b_{4i} x_i^2\right)x_4 + c_4x_2x_3x_4^2.
\end{eqnarray*}
The coefficients $b_{ii}$ need to be different from the rest but the remaining three may be equal.\footnote{This vector field is not equivariant for $\Gamma = \langle \kappa_2,\kappa_3,\kappa_4 \rangle$ used for the $B_2^+$ cycle (see \cite{CastroLohse}, Appendix A). However, when restricted to the planes $P_{1k}$ the two vector fields do coincide. The information at linear level for the dynamics coincides with that for the $B_2^+$ cycle.} 

Choose $a_i>0$ so that the origin is unstable. We can then choose $b_{11}^2-4a_1c_1>0$ as well as $b_{11}<0$ and $c_1<0$ to ensure the existence of only two distinct equilibria on the $x_1$-axis. The remaining coefficients may be chosen to provide the necessary saddle-sink connections.

Therefore, this normal form supports an $A_2^+$ cycle and two of these can be put together in an $(A_2^+,A_2^+)$ network. 

\subsection*{A.2 \hspace{.2cm} The $(A_3^-,A_3^-)$ and $(A_3^-,A_3^-,A_4^-)$ networks}

For the construction of a cycle of type $A_3^-$ we look at
$$
\Gamma = \langle \kappa_{12},\kappa_{13},\kappa_{34} \rangle,
$$
which contains all $\kappa_{ij}$ and $-\Id \in \Gamma$, so that the same symmetry supports an $A_4^-$ cycle.
A $\Gamma$-equivariant vector field for this symmetry is
$$
\dot{x_j} = a_jx_j + \left(\sum_{i=1}^{4} b_{ji} x_i^2\right)x_j + c_jx_1x_2x_3x_4 x_j.
$$
When restricted to a plane, say $P_{12}$, we obtain
\begin{eqnarray*}
\dot{x_1}& = & a_1x_1 + \left(\sum_{i=1}^{2} b_{1i} x_i^2\right)x_1\\
\dot{x_2}& = & a_2x_2 + \left(\sum_{i=1}^{2} b_{2i} x_i^2\right)x_2.
\end{eqnarray*}
Note that this does not support a heteroclinic cycle of type $A_2^+$ as the two equilibria on the $x_1$-axis would be related by symmetry:
$$
a_1x_1+b_{11}x_1^3=0 \qquad \Leftrightarrow \qquad x_1=0 \quad \mbox{or} \quad x_1= \pm \sqrt{-a_1/b_{11}},
$$
where $\Gamma. (\sqrt{-a_1/b_{11}},0,0,0) = \{(\sqrt{-a_1/b_{11}},0,0,0) , (-\sqrt{-a_1/b_{11}},0,0,0) \}$ is the corresponding group orbit.

Coefficients can be chosen to ensure a saddle-sink connection in $P_{12}$ in a standard way. Two cycles of type $A_3^-$ can be put together in a network in a way analogous to that used for the $B_3^-$ cycle. Again the information at linear level required for the dynamics is the same as that obtained for the $(B_3^-,B_3^-)$ network.

\end{document}